\documentclass[a4paper]{amsart}

\usepackage[all]{xy}
\usepackage{a4wide}
\usepackage{verbatim}
\usepackage{amsmath}
\usepackage{amssymb}
\usepackage{amsthm}
\usepackage{latexsym}
\usepackage{enumerate}
\usepackage{prettyref}
\usepackage{hyperref}
\usepackage{bm}
\usepackage[linewidth=1pt]{mdframed}
\usepackage{mathrsfs}
\usepackage[lite]{amsrefs}

\usepackage{algorithmicx}
\usepackage[ruled]{algorithm}
\usepackage{algpseudocode}
\usepackage{algc}
\newenvironment{alginc}[1][pseudocode]{\medskip\algsetlanguage{#1}\begin{algorithmic}[0]}{\end{algorithmic}\medskip}

\usepackage{tikz}
\usetikzlibrary{matrix,arrows}
\usepackage{tikz-cd}
\usepackage{xypic}
\usepackage{float}
\usepackage{pstricks}

 \newcommand{\F}{{\mathbb{F}}}

\newtheorem{theorem}{Theorem}[section]
\newrefformat{thm}{\hyperref[{#1}]{Theorem~\ref*{#1}}}
\newtheorem{definition}[theorem]{Definition}
\newrefformat{def}{\hyperref[{#1}]{Definition~\ref*{#1}}}

\newrefformat{lem}{\hyperref[{#1}]{Lemma~\ref*{#1}}}

\newrefformat{prop}{\hyperref[{#1}]{Proposition~\ref*{#1}}}

\newrefformat{cor}{\hyperref[{#1}]{Corollary~\ref*{#1}}}
\newtheorem{remark}[theorem]{Remark}
\newtheorem{Condition}[theorem]{Condition}
\newrefformat{rem}{\hyperref[{#1}]{Remark~\ref*{#1}}}
{
\newtheorem{examplecore}[theorem]{Example}}
\newrefformat{ex}{\hyperref[{#1}]{Example~\ref*{#1}}}

\newcommand{\op}{\operatorname}

\newcommand*{\Homol}{\operatorname{H}}
\newcommand*{\Farrell}{\widehat{\operatorname{H}}}
\newcommand*{\Steinberg}{\widetilde{\operatorname{H}}}
\newcommand{\Z}{\mathbb{Z}}
\newcommand{\Afour}{\mathcal{A}_4}

\newcommand{\Sfour}{{\mathcal{S}}_4}

\newcommand{\Dtwo}{{\mathcal{D}}_2}

\newcommand{\Dfour}{{\mathcal{D}}_4}

\newcommand{\CtwoSq}{(\Z/2\Z)^2}
\newcommand{\Ctwo}{\Z/2\Z}
\newcommand{\Deight}{{\mathcal{D}}_8}
\newcommand{\CtwoCube}{(\Z/2\Z)^3}
\newcommand{\Cfour}{\Z/4\Z}

\begin{document}

\title[Farrell--Tate cohomology of \texorpdfstring{${\rm PSL}_4(\mathbb{Z})$}{PSL4Z} and Discrete Morse theory]{The 2-torsion in the Farrell--Tate cohomology of \texorpdfstring{${\rm PSL}_4(\mathbb{Z})$}{PSL4Z}, \mbox{and torsion subcomplex reduction via discrete Morse theory}}

\author{Alexander D. Rahm$^1$, Bui Anh Tuan$^{2,3}$,  and Matthias Wendt$^4$}

\date{\today}

\address{$^1$Laboratoire de math\'ematiques GAATI, Universit\'e de la Polyn\'esie Fran\c{c}aise, BP 6570, 98702 Faa'a, French Polynesia\\ Alexander.Rahm@upf.pf , ORCID: 0000-0002-5534-2716, https://gaati.org/rahm}
\address{$^2$Faculty of Mathematics and Computer Science, University of Science, Ho Chi Minh City, Vietnam}
\address{$^3$Vietnam National University, Ho Chi Minh City, Vietnam}
\address{$^4$Matthias Wendt, Fachgruppe Mathematik und Informatik, Bergische Universit\"at Wuppertal, Gauss\-strasse 20, 42119 Wuppertal, Germany}

\subjclass[2020]{11F75, 20J06}
\keywords{Cohomology of arithmetic groups, Farrell--Tate cohomology, projective special linear group of rank 4, discrete Morse theory, torsion subcomplex reduction}

\begin{abstract}
In the present paper, we use an equivariant modification of discrete Morse theory to provide a new implementation of torsion subcomplex reduction for arithmetic groups. This leads both to a simpler algorithm as well as runtime improvements. To demonstrate the technique, we compute the mod 2 Farrell--Tate cohomology of $\op{PSL}_4(\mathbb{Z})$.
\end{abstract}

\maketitle
\setcounter{tocdepth}{1}

\section{Introduction}

In this paper, we compute the mod 2 Farrell--Tate cohomology of $\op{PSL}_4(\Z)$, the rank 4 projective special linear group over the integers. Above the virtual cohomological dimension (which is 6 according to the formula of Borel and Serre \cite{Borel:Serre}), Farrell--Tate cohomology (with mod $\ell$ coefficients) agrees with group cohomology, thus we also compute the dimension of the mod $2$ group cohomology $\Homol^q(\op{PSL}_4(\Z);\thinspace \F_2)$. The main result is the following:

\begin{theorem}\label{result}
The mod-2 Farrell--Tate cohomology of $\op{PSL}_4(\Z)$ has the following dimensions over $\F_2$ in degrees $0 \leq q \leq 6$:
$$
\begin{array}{|l|c|c|c|c|c|c|c|c|c|c|c|c|c|c|c|c|c|c|c|c|c|}
 \hline q  & 0 & 1 & 2 & 3 & 4 & 5 & 6\\ \hline
 \dim_{\F_2}\Farrell^q(\op{PSL}_4(\Z);\thinspace \F_2) & 1 & 0 & 5 & 7 & 7 & 17 & 27 \\ \hline
\end{array}
$$
\normalsize
In degrees $7 \leq q \leq 42$,
the dimension of the Farrell--Tate cohomology
$\dim_{\F_2}\Farrell^q(\op{PSL}_4(\Z);\thinspace \F_2)$
and hence the dimension of the group cohomology
$\dim_{\F_2}\Homol^q(\op{PSL}_4(\Z);\thinspace \F_2)$,
equals the number
$\dim_{\F_2}E_2^{0,q}$ computed in Table~\ref{7-42}.
More generally, the Hilbert--Poincar\'e series for the mod 2 Farrell--Tate cohomology of $\op{PSL}_4(\Z)$ is the following:
\[
\sum\limits_{q=0}^\infty \dim_{\F_2}\Farrell^q(\op{PSL}_4(\Z);\thinspace \F_2)\cdot T^q = \frac{2T^{11}-T^{10}-4T^9-3T^8+3T^7+5T^6-3T^4+T^3+4T^2-T+1}{(T^2+T+1)^2(T+1)(T-1)^4}
\]
\end{theorem}
Since with field coefficients, cohomology is precisely the dual of homology~\cite{Hatcher}*{chapter 3, section ``Cohomology of Spaces''}, we have the same dimensions for the homology groups $ \Homol_q(\op{PSL}_4(\Z);\thinspace \F_2)$ $\cong \Homol^q(\op{PSL}_4(\Z);\thinspace \F_2)$. Together with the mod $3$ and mod $5$ cohomology results of the authors~\cite{psl4z}, the Universal Coefficient Theorem gives us a tight control over the integral group homology above the virtual cohomological dimension (vcd), because in those degrees only $2$-, $3$- and $5$-primary torsion can occur, other primes not being present in the orders of elements of the isotropy groups. In the degrees below the vcd, there are the results of Dutour, Ellis and Sch\"urmann~\cite{Dutour:Ellis:Schuermann}, hence a quite complete picture of $\Homol_*(\op{PSL}_4(\Z);\thinspace \Z)$ is achieved.

The computations of a contractible cell complex with a proper action of $\op{PSL}_4(\Z)$ date back to perfect forms calculations by {\v S}togrin in 1974 \cite{Stogrin} and independently and notably those of Lee and Szczarba in 1978 \cite{LeeSzczarbaTorsion}; such a cell complex was available on computers just after the turn of the millenium~\cites{Ash:Gunnells:McConnell, Dutour:Ellis:Schuermann, Elbaz-Vincent:Gangl:Soule, Elbaz-Vincent:Gangl:Soule2}; one of the authors of this paper (A.R.) did pick up the task of computing the cohomology of $\op{PSL}_4(\Z)$ at small primes as his PhD thesis project in 2007, but in order to deal with prime numbers present in the orders of cell stabilizers, a whole new set of sophisticated tools had to be developed over the years, documented in a series of publications by the authors. The final tool was recently contributed by Ellis~\cite{HAP}. That is why the mod $2$ cohomology of $\op{PSL}_4(\Z)$ has only been computed in the present paper, 47 years after the mod $2$ cohomology computation for $\op{SL}_3(\Z)$ by Soul\'e~\cite{Soule}.

In principle, our results could also be refined to compute the ring structure on cohomology $\Homol^*(\op{PSL}_4(\Z);\F_2)$ above the virtual cohomological dimension. From our computations it follows that the equivariant spectral sequence computing mod 2 cohomology degenerates at the $E_2$-page, and for degrees above the virtual cohomological dimension, all information is contained in the column $E^{0,q}_2$. This means that the natural restriction map
\[
\Farrell^*(\op{PSL}_4(\Z);\F_2)\to \bigoplus_{\sigma\in\mathcal{X}}\Homol^*(\Gamma_\sigma;\F_2)
\]
from $\op{PSL}_4(\Z)$ to the stabilizer groups $\Gamma_\sigma$ on the reduced $\ell$-torsion subcomplex is actually injective above the virtual cohomological dimension. In particular, any cup products of classes could in principle be computed as cup products in the cohomology rings of finite stabilizer subgroups (which are known, see Table~\ref{stabilizers-and-their-cohomology-rings}).

% In Sections~\ref{sec:torsionsubcomplex} to~\ref{sec:farrelltate}, we detail our computations, and in Section~\ref{sec:Steinberg}, we point out their consequences on Steinberg homology.
\begin{table}
\caption{The $E_2$-page of the equivariant spectral sequence converging to the mod $2$ Farrell--Tate cohomology of PSL$_4(\Z)$ is, in degrees $q\geq 2$, concentrated in the single column $p = 0$.
We record the following dimensions over $\F_2$ for $E_2^{0,q}$ for $7 \leq q \leq 42$:}
\label{7-42}

\footnotesize
$$
\begin{array}{|l|c|c|c|c|c|c|c|c|c|c|c|c|c|c|c|c|c|c|c|c|c|}
 \hline q             &  7 &  8 &  9 & 10 & 11 & 12 & 13 & 14 & 15 & 16 & 17 & 18 & 19 & 20 & 21 & 22 & 23\\ \hline
 \dim_{\F_2}E_2^{0,q} & 31 & 50 & 67 & 78 & 107& 134& 153& 195& 233& 263& 319& 371& 413& 486& 553& 610& 701
\\ \hline
\end{array}
$$
$$
\begin{array}{|l|c|c|c|c|c|c|c|c|c|c|c|c|c|c|c|c|c|c|c|c|c|}
 \hline q             & 24  & 25 & 26 & 27 & 28 & 29 & 30 & 31 & 32 & 33 & 34 & 35 & 36 & 37 \\ \hline
 \dim_{\F_2}E_2^{0,q} & 786 & 859& 971 & 1075 & 1167 & 1301 & 1427 & 1539 & 1698 & 1847 & 1982 & 2167 & 2342 & 2501
\\ \hline
\end{array}
$$

$$
\begin{array}{|l|c|c|c|c|c|c|c|c|c|c|c|c|c|c|c|c|c|c|c|c|c|}
 \hline q             & 38 & 39 & 40 & 41 & 42\\ \hline
 \dim_{\F_2}E_2^{0,q} & 2715 & 2917 & 3103 & 3347 & 3579
\\ \hline
\end{array}
$$
\normalsize
\end{table}

For carrying out the computation of mod 2 group cohomology, we provide in this paper a further tool, which we hope can help other mathematicians to compute small prime cohomology of more complicated arithmetic groups. Namely, in Section~\ref{sec:morse}, we adapt discrete Morse theory for the computation of mod $\ell$ Farrell--Tate cohomology of a complex of groups. The modification we propose produces very small complexes of groups describing, for a given complex of groups, a bi-graded complex of $\F_\ell$-vector spaces which is quasi-isomorphic to the $E_1$-page of the associated equivariant spectral sequence. In particular, the reduced $E_1$-page obtained from discrete Morse theory completely recovers the original $E_2$-page of the equivariant spectral sequence. % Our modification deals with the fact that we are not just interested in the cohomology of a cell complex, but actually the equivariant cohomology of a complex of groups.
If we are interested specifically in Farrell--Tate or group cohomology at a single prime number $\ell$, our discrete Morse theory method is an improvement over the previously used $\ell$-torsion subcomplex reduction methods. Because our $\ell$-torsion modification of discrete Morse theory produces significantly smaller cell complexes, it substantially simplifies computations with the equivariant spectral sequence. % The Morse-theoretic torsion subcomplex reduction provides a description of the (mod $\ell$) $E_1$-page of the equivariant spectral sequence which yields again the original $E_2$-page, hence the computations provide descriptions of Farrell--Tate cohomology in all degrees.

We do not aim here at developing a very general equivariant version of discrete Morse theory, as this has already been done in the literature~\cites{Freij, Yerolemou:Nanda}. Rather, we want to provide a specific tool for studying Farrell--Tate or group cohomology of (arithmetic) groups at a single prime number~$\ell$. Our approach differs from the one of Yerolemou and Nanda~\cite{Yerolemou:Nanda}, in that we can allow infinite groups where \cite{Yerolemou:Nanda} requires the group under study to be finite. Our approach differs from Freij's approach~\cite{Freij} in that the focus on a single prime $\ell$ and the use of $\ell$-torsion subcomplexes can simplify the description of the equivariant spectral sequence even further (while of course losing some information not visible in $\F_\ell$-cohomology).

\subsubsection*{Structure of the paper}
We first discuss how torsion subcomplex reduction can be realized by a version of discrete Morse theory taking into account fusion control for stabilizer subgroups in Section~\ref{sec:morse}. The example of the reduced 2-torsion subcomplex for ${\rm PSL}_4(\mathbb{Z})$ is discussed in Section~\ref{sec:torsionsubcomplex}. In Section~\ref{sec:d1}, we assemble the $d_1$ differential of the equivariant spectral sequence, which was the major technical task of this paper. The spectral sequence computations for the mod 2 Farrell--Tate cohomology of ${\rm PSL}_4(\mathbb{Z})$ are carried out in Section~\ref{sec:farrelltate}. In Section~\ref{sec:hpseries}, we compute the Hilbert--Poincar\'e series of mod 2 Farrell--Tate cohomology of $\op{PSL}_4(\Z)$, and in Section~\ref{sec:Steinberg}, we point out the consequences of our computations on Steinberg homology.

\subsubsection*{Acknowledgements}
We would like to express our special thanks to Graham Ellis, who, by building a new feature into his HAP package in GAP upon our request, made it possible to remove the ambiguities in assembling our $d_1$ differential.
We would like to heartily thank Bill Allombert (Pari/GP Development Headquarters), who helped us coding the GP script which formerly did run through all possibilities of assembling our $d_1$ differential, and still checks that the assembly is correct and produces the sanity check diagrams.
We would like to warmly thank Simon King for having released his Mod-$p$ Group Cohomology Package and answered our questions about it, and Matthias K\"oppe and especially Dmitrii Pasechnik for a lot of help on getting this package to run again in the contemporary version of SAGE.
We are indebted to Peter Patzt for pointing us in Section~\ref{sec:Steinberg} to the relevant theorems of Lee and Szczarba.
Finally, we would like to thank Ethan Berkove for discussing the equivariant spectral sequence with us.
We would like to acknowledge financial support by the ANR grant MELODIA (ANR-20-CE40-0013).
Among the authors, Tuan Anh Bui is fully funded by Vietnam National University - Ho Chi Minh City (VNU-HCM) under grant number C2023-18-01.

\begin{table} \caption{Numbers of cells in the $2$-torsion subcomplex for PSL$_4(\Z)$ before reduction, obtained after rigid facets subdivsion of the studied cell complex, sorted into isomorphism types of their stabilizers. Here, $G_1$ and $G_2$ are the groups (288,1026) and (96,227), respectively, in GAP's SmallGroups library. They are given by non-trivial group extensions
$1 \to \Afour \times \Afour \to G_1 \to \Ctwo \to 1$ and $1 \to G_0 \to G_2 \to \Ctwo \to 1$ for $1 \to (\Z/2\Z)^4 \to G_0 \to \Z/3\Z \to 1$.}
\label{2-torsion-cell-numbers}
\hspace{-3.1em}
\footnotesize
\[
\begin{array}{|l|r|r|r|r|r|r|r|r|l|} \hline %&&&&&&&&&\\
\text{Stabilizer type} & \Afour &    G_1 & \CtwoSq& \Sfour &\Ctwo&\Deight & \CtwoCube & G_2 & \Cfour \\
\hline %&&&&&&&&& \\
\text{Vertices}    & 2      &       1   &     5  &  4     &  1  &   2    &       1   &         1        &  0  \\
\text{Edges}       & 2      &       0   &    24  &  2     &101  &   5    &       1   &         0        &  5  \\
\text{2-cells}     & 0      &       0   &    27  &  0     &326  &   0    &       0   &         0        &  4  \\
\text{3-cells}     & 0      &       0   &     8  &  0     &340  &   0    &       0   &         0        &  0  \\
\text{4-cells}     & 0      &       0   &     0  &  0     &116  &   0    &       0   &         0          &  0  \\ \hline
\end{array}
\]
\normalsize
\end{table}

\section{Torsion subcomplex reduction via discrete Morse theory}
\label{sec:morse}

Let $\Gamma$ be a discrete group acting properly on a cell complex $\mathcal{X}$. The object which we shall reduce using discrete Morse theory, is the $\ell$-torsion subcomplex for the $\Gamma$-action on $\mathcal{X}$, at a prime number~$\ell$.
In order to compute the mod $\ell$ Farrell--Tate cohomology of a group from its proper action on a contractible cell complex, it is enough to study the action on the $\ell$-torsion subcomplex, see~\cite{TransAMS}.
Our reduction of the latter brings further arithmetic groups within reach.

\begin{definition}\label{subcomplex}
 Let $\Gamma$ be a discrete group acting properly on a cell complex $\mathcal{X}$. For $\ell$ a prime number, the \emph{$\ell$-torsion subcomplex} is the set of all cells of $\mathcal{X}$ whose $\Gamma$-stabilizers contain some element(s) of order $\ell$.
\end{definition}

For the $\ell$-torsion subcomplex to be guaranteed to be a cell complex, and to consist only of fixed points of order-$\ell$-elements (so to coincide with the $\ell$-singular part), we need a rigidity property: We want each cell stabilizer to fix its cell pointwise. In theory, it is always possible to obtain this rigidity property via the barycentric subdivision. In practice, to avoid a memory stack overflow, we need to use the cell subdivision algorithms ``rigid facets subdivision'' and ``virtually simplicial subdivision'' provided in previous work of the authors~\cite{psl4z}, because the barycentric subdivision of an $n$-dimensional cell complex can multiply the number of cells by $(n + 1)!$.

The rigid cell complexes returned by rigid facets subdivison are rather big, compared to the non-rigid complexes we start with (-see for example table 1 in \cite{psl4z} for information on the full cell complex; for the 2-torsion subcomplex, we reprint Table~\ref{2-torsion-cell-numbers} here from that paper with the kind permission of the authors;-).

To reduce the size of the complexes (as well as the size of the resulting spectral sequences), we now discuss a variant of discrete Morse theory taking into account the additional information of stabilizer subgroups: the gradient flow is only allowed to proceed in directions where the cohomology of the stabilizer subgroups does not change, as specified in Condition~\ref{cancelling-conditions} below. In our computation, the number of critical cells is significantly smaller than the number of cells before reduction. Moreover, the gradient flow restriction on stabilizer subgroups implies that the output of the modified discrete Morse algorithm produces a complex quasi-isomorphic to the $E_1$-page of the equivariant spectral sequence, thus recovering the correct $E_2$-page, see Theorem~\ref{correctness-of-algorithm}.

In order to reduce the $\ell$-torsion subcomplex at a prime number $\ell$ (cf. Definition~\ref{subcomplex} above),
our conditions for cancelling an $n$-cell $\tau$ against one of its boundary $(n-1)$-cells $\sigma$,
are the following ones, already established in previous work~\cite{Tbilisi}.

 \begin{Condition}\label{cancelling-conditions} As conditions for cancelling an $n$-cell $\tau$ against one of its boundary $(n-1)$-cells $\sigma$, and hence for admitting an arrow from $\sigma$ to $\tau$ into the discrete gradient vector field, we impose: \begin{itemize}
\item No higher-dimensional cells of the $\ell$-torsion subcomplex touch $\tau$,
\item The interior of two cells connected by a path of the discrete gradient vector field do not contain two points which are on the same orbit under the group action,
\item The inclusion of the stabilizer of $\tau$ into the stabilizer of $\sigma$ induces an isomorphism on mod $\ell$ cohomology.
\end{itemize} \end{Condition}
Practical criteria for checking if an inclusion of stabilizers induces an isomorphism on mod $\ell$ cohomology are given in~\cite{Tbilisi}*{Condition B'}, based on control of $\ell$-fusion. As discussed in \cite{Tbilisi}, the mod $\ell$ equivariant cohomology is not affected by such a cancellation.

\subsection{Construction of a discrete gradient vector field}

\begin{figure}
  \caption{Example: The \texorpdfstring{$2$}{2}-torsion subcomplex for SL\texorpdfstring{$_3(\mathbb{Z})$}{(3,\textbf{Z})}}
\label{nice-vector-field}
\begin{mdframed}
On the $2$-torsion subcomplex of the cell complex described by Soul\'e~\cite{Soule},
 obtained from the action of SL$_3(\mathbb{Z})$ on its symmetric space (diagram from previous work of one of the authors~\cite{Tbilisi}), we draw here a discrete gradient vector field (in bold arrows).
\begin{center}
 \scalebox{0.7} % Change this value to rescale the drawing.
 {
 \includegraphics{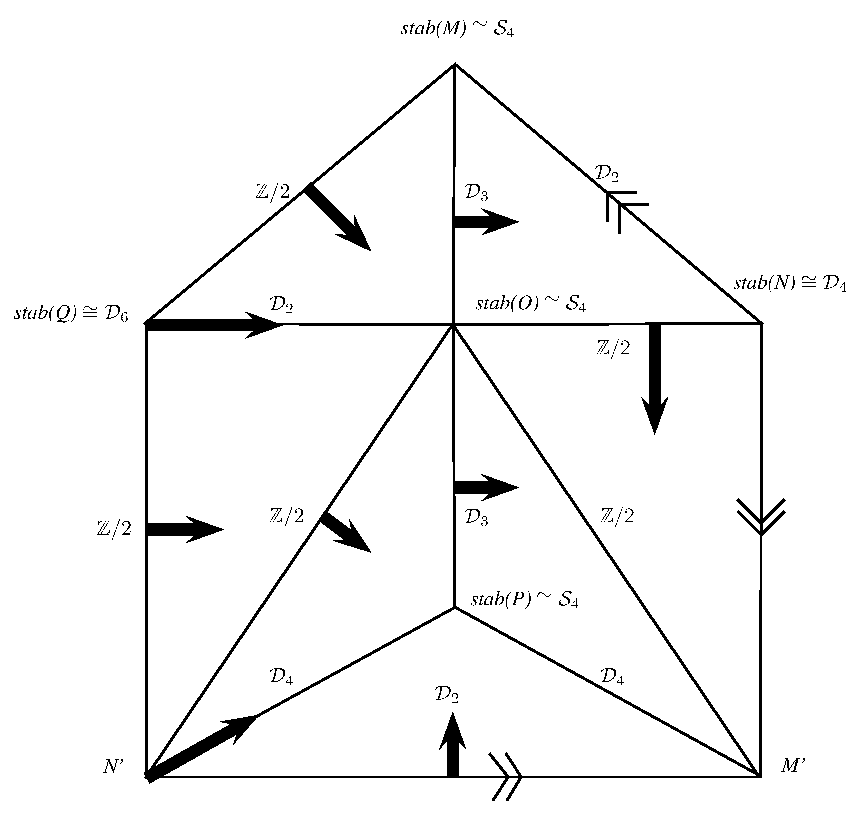}
}
\end{center}
Here, the three edges $NM$, $NM'$ and $N'M'$ have to be identified as indicated by the double arrows.
All of the seven triangles belong with their interior to the $2$-torsion subcomplex,
each with stabilizer $\Z/2$, except for the one which is marked to have stabilizer $\Dtwo$.
Along the discrete gradient vector field indicated by the simple arrows, we reduce this subcomplex to
\begin{center}
%  \scalebox{1} % Change this value to rescale the drawing.
% {
\begin{pspicture}(-1.9,-0.2)(8.0,0.3)
        \uput{0.1}[90](2.0,0.0){ $\Sfour$}
        \psdots(2,0.0)
        \psline(2,0.0)(6,0.0)
                        \uput{0.1}[90](3.0,0.0){ $\Z/2$}
        \uput{0.1}[90](4.0,0.0){ $\Sfour$}
        \psdots(4,0.0)
                        \uput{0.1}[90](5.0,0.0){ $\Dfour$}
               \uput{0.1}[90](6.0,0.0){ $\Sfour$}
        \psdots(6,0.0)
\end{pspicture}
% }
\end{center}
which is the geometric realization of Soul\'e's diagram of cell stabilizers.
This yields the mod~$2$ Farrell--Tate cohomology as specified by Soul\'e~\cite{Soule}.
\end{mdframed}
\end{figure}

\begin{figure}
\caption{We can also put an essentially different discrete gradient vector field on the $2$-torsion subcomplex from Figure~\ref{nice-vector-field}:}
\label{ugly-vector-field}
\begin{mdframed}
\begin{center}
\scalebox{0.7} % Change this value to rescale the drawing.
{
 \includegraphics{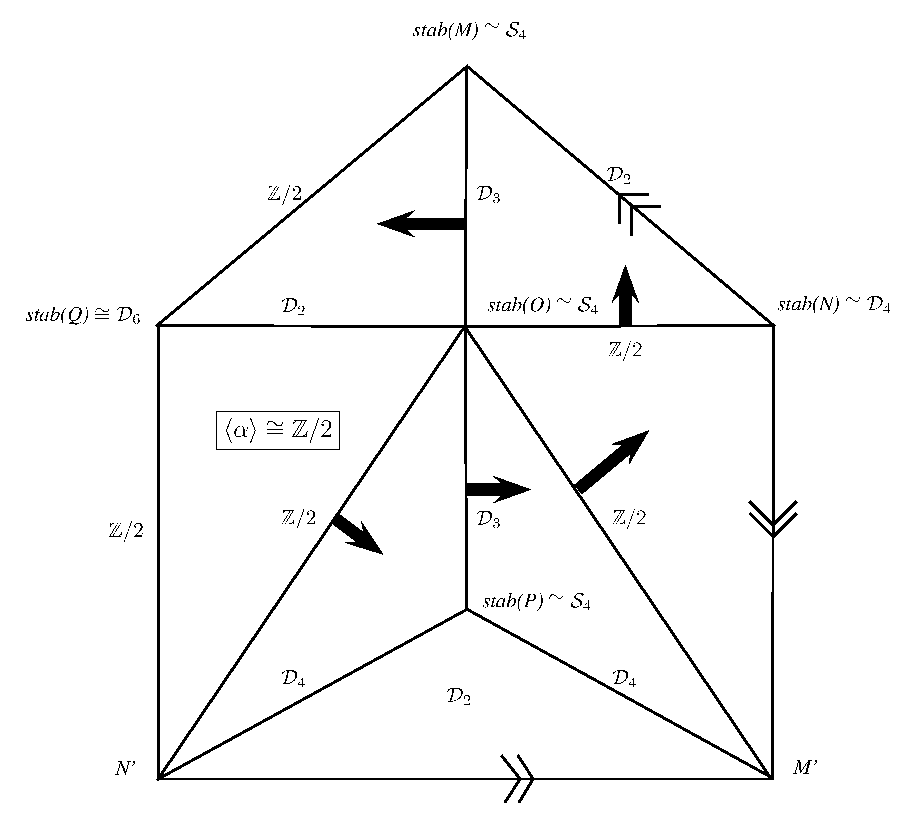}
}
\end{center}
Now, the $2$-cells are not being cut off from the boundary, but they are cancelled against edges in the interior of the subcomplex.
As a consequence, the stabilizer $\langle \alpha \rangle$ of one of the two triangles which are not cancelled, needs to have monomorphisms into the stabilizers of each of the edges which form the new boundary of its cell after cancellation: $N'P$, $PM'$, $M'N$, $NM$, $MQ$, $QO$.
These monomorphisms are given by the inclusion of the generator $\alpha$ into the stabilizer of the first cancelled edge $N'O$, which is provided by the rigidity of the cell complex; then the isomorphism of stabilizers which allows to put the arrow starting from that edge, then the inclusion into the stabilizer of the edge $PO$ (again by rigidity); then the isomorphism on mod $2$ cohomology which allows to put the arrow starting from the edge $PO$, and so on.
\end{mdframed}
\end{figure}

In this section, we describe an algorithm constructing a discrete vector field for a discrete Morse function in the sense of Forman~\cite{Forman}, with the extra property that $\ell$-fusion is controlled along its arrows.
We obtain it by inserting a fusion control condition into the \emph{Discrete vector field on a regular CW-complex} algorithm of Ellis \& Hegarty~\cite{Ellis:Hegarty}, see Algorithm~\ref{ALGdiscretevectorfield} below. For this purpose, we need to recall the following definitions from the literature.

\begin{definition}
A discrete vector field is said to be \emph{admissible} if it contains no circuits and no chains that extend infinitely to the right. When the underlying CW-complex is finite, a discrete vector field is admissible if it contains no circuits.\footnote{This is the condition that the discrete vector field is the gradient vector field of a discrete Morse function.}
\end{definition}

\begin{definition}
We say that an admissible discrete vector field is \emph{maximal} if it is not possible to add an arrow while retaining admissibility. A cell in the CW-complex is said to be \emph{critical} if it is not involved in any arrow.
\end{definition}

In order to take the cell stabilizers into account, we need to equip the CW-complex with an extra structure, ensured by the terminology ``simple complex of groups'' introduced by Bridson and Haefliger \cite{Bridson:Haefliger}*{II.12.11 on page 375}:

\vbox{
\begin{definition}
Let $X$ be a regular CW-complex. A \emph{simple complex of groups} on $X$ is a set $\{ G_\sigma, \psi_{\tau\sigma}\}_{\sigma, \tau\sigma}$, where
\begin{itemize}
\item for each cell $\sigma\in X$, we have a group $G_\sigma$,
\item for each cell $\tau$ in the boundary $\partial \sigma$ of a cell $\sigma$ of $X$, we have an injective homomorphism $\psi_{\tau\sigma}:G_{\sigma}\to G_{\tau}$
such that if $\tau \in \partial \sigma$ and $\sigma \in \partial \rho$, then $\psi_{\tau\rho} =\psi_{\tau\sigma}\psi_{\sigma\rho}$.
\end{itemize}
\end{definition}
}

\begin{algorithm}[h]
\caption{The \emph{Discrete vector field on a regular CW-complex} algorithm of Ellis \& Hegarty~\cite{Ellis:Hegarty}, with our fusion control condition inserted.}
\label{ALGdiscretevectorfield}

\begin{alginc}

\State \bf Input: \rm  A finite regular CW-complex $Y$ carrying a simple complex of groups.

\State \bf Output: \rm  A  maximal admissible discrete vector field on $Y$.

\Procedure{}{}

\State Partially order the cells of $Y$ in any fashion.

\State At any stage of the algorithm each cell  will have precisely one of the following three states:
\State (i) {\em critical}, (ii)
{\em potentially critical}, (iii) {\em non-critical}.

\State Initially deem all cells of $Y$ to be potentially critical.

\While{there exists a potentially critical cell}

\While{there exists a pair of potentially critical cells $s,t$
  such
 that:  dim$(t)=$dim$(s)+1$;  $s$ lies in the boundary of $t$;  no other
potentially critical cell of dimension dim$(s)$ lies in the boundary of $t$,
and Condition~\ref{cancelling-conditions} is satisfied for the pair $(s,t)$;}

\State Choose such a pair $(s,t)$ with $s$ minimal in the given partial ordering.

\State Add the arrow $s\rightarrow t$ and deem $s$ and $t$ to be  non-critical.

\EndWhile

\If{there exists a potentially critical cell}

\State Choose a minimal potentially critical cell and deem it
 to be  critical.

\EndIf

\EndWhile

\EndProcedure
\end{alginc}
\end{algorithm}

\subsection{Critical cells and the equivariant spectral sequence}

Once a maximal admissible discrete vector field has been constructed, we can reduce the torsion subcomplex by cancelling the two cells involved in an arrow against each other, such that only critical cells remain.
For this purpose, we have to take care of what we shall call the boundary heritage.

\begin{definition}
The \emph{boundary heritage} of an $n$-cell which has some of its boundary cells cancelled by Morse reduction, consists for each cancelled boundary $(n-1)$-cell $\sigma$ of the non-cancelled boundary cells of the $n$-cell at the other end of the path leading to $\sigma$ via the discrete gradient vector field.
\end{definition}

We can efficiently determine the boundary heritage by adapting Forman's construction of the Morse differential \cite{Forman} to simple complexes of groups, as follows.
\\
We take an $(n+1)$-cell $\sigma$. For each $n$-cell $\tau\in \partial\sigma$, we can take the sum over all paths $p$ (with signs coming from orientations of cells, the orientations being obtained as described in \cite{psl4z}*{algorithm~3}) starting at $\tau$, of the restriction maps $\op{H}^\bullet(G_{\rho_p})\to\op{H}^\bullet(G_\sigma)$. Here, $G_{\rho_p}$ denotes the group at the end of the path $p$, and the group homomorphism is the composition of the isomorphism $\Homol^*(G_\tau)\cong \Homol^*(G_{\rho_p})$ associated to the path $p$ and the inclusion $G_\sigma\hookrightarrow G_\tau$.

\bigskip

Our above reflections culminate in the following theorem, which allows us to simplify the beginning of the equivariant spectral sequence.

\begin{theorem} \label{correctness-of-algorithm}
  Let $\Gamma$ be a discrete group acting cellularly, properly and cocompactly on a cell complex $\mathcal{X}$, and fix some prime $\ell$. We denote by $E^{p,q}_1(\mathcal{X})$ the $E_1$-page of the mod $\ell$ equivariant spectral sequence computing the mod $\ell$ Farrell--Tate cohomology of the $\Gamma$-complex $\mathcal{X}$.

  Assume that the associated $\ell$-torsion subcomplex $\mathcal{X}_\ell$ is again a $\Gamma$-complex. Equip $\mathcal{X}_\ell$ with a maximal admissible (i.e., gradient) discrete vector field, as constructed by Algorithm~\ref{ALGdiscretevectorfield}. Then we consider the complex of graded $\mathbb{F}_\ell$-vector spaces
  \[
  \left(\bigoplus_{\sigma\in\left(\mathcal{X}_\ell/\Gamma)\right)^{(p)} \thinspace{\rm critical!}}\Homol^q(\Gamma_\sigma;\thinspace\F_\ell);\thinspace d'_1\right)
  \]
  obtained as direct sum of mod $\ell$ cohomologies of stabilizers of \emph{critical} $p$-cells in the $\ell$-torsion subcomplex $\mathcal{X}_\ell$. The differential $d_1'$ is given by composition of the original $d_1$ differential in $E_1^{p,q}(\mathcal{X})$ with the isomorphisms to the corresponding stabilizer groups in the boundary heritage, as described above. 
  
  Then the natural map
  \[
  \bigoplus_{\sigma\in\left(\mathcal{X}_\ell/\Gamma\right)^{(p)} \thinspace{\rm critical!}} \Homol^q(\Gamma_\sigma;\thinspace\F_\ell)\to \bigoplus_{\sigma\in\left(\mathcal{X}/\Gamma\right)^{(p)}} \Homol^q(\Gamma_\sigma;\thinspace\F_\ell)=E^{p,q}_1(\mathcal{X})
  \]
  is a quasi-isomorphism of complexes of graded $\mathbb{F}_\ell$-vector spaces. In particular, its cohomology is naturally isomorphic to the $E_2$-page $E^{p,q}_2(\mathcal{X})$ of the equivariant spectral sequence associated to the $\Gamma$-complex $\mathcal{X}$.
%  Reducing the associated $\ell$-torsion subcomplex consisting of finitely many orbits along a discrete gradient vector field constructed by Algorithm~\ref{ALGdiscretevectorfield}, and attaching the boundary heritage to replace lost boundary cells, induces an isomorphism on the $E_2$ page of the equivariant spectral sequence for the group action $G\looparrowright\mathcal{X}$, and hence on equivariant mod $\ell$ cohomology.
\end{theorem}

\begin{proof}
  We already know that inclusion $\mathcal{X}_\ell\to\mathcal{X}$ induces an isomorphism on the $\Gamma$-equivariant mod~$\ell$ (Farrell--Tate) cohomology, see \cite{Tbilisi}. So we only need to show that reducing to the critical cells of $\mathcal{X}_\ell$ doesn't change the cohomology of the $E_1$-page for the $\Gamma$-complex $\mathcal{X}_\ell$.
  
As we only have to deal with finitely many orbits modulo the group action, we can break down the reduction along the discrete vector field into finitely many steps of cancelling an $n$-cell against one of its boundary cells.
\begin{itemize}
 \item For a path of $n$- and $(n-1)$-cells through the discrete gradient vector field which starts at the boundary of the orbit space (see Figure~\ref{nice-vector-field}), we can cancel one pair of cells after another without affecting the boundary of other cells, and each time, such a cancellation lowers the rank of the $d_1^{n-1,*}$-differential as much as the rank of the $E_1^{n-1,*}$- and $E_1^{n,*}$-terms of the equivariant spectral sequence. Essentially, for each arrow in the discrete vector field relating the $n$-cell $\sigma$ and its boundary cell $\tau$, we are splitting off an obviously contractible complex $0\to \Homol^*(\Gamma_\tau;\thinspace\F_\ell)\to \Homol^*(\Gamma_\sigma;\thinspace\F_\ell)\to 0$ from the $E_1$-page. Hence the $E_2$-page remains invariant, and thanks to our condition that no higher-dimensional cells of the $\ell$-torsion subcomplex touch the concerned cells, no higher-degree differentials have targets on them, and hence the equivariant mod~$\ell$ cohomology remains invariant as well.
\item For a path through the discrete gradient vector field which is contained in the interior of the orbit space (see Figure~\ref{ugly-vector-field}), keeping track of the boundary heritage and attaching it to the neighbouring cells after each cancellation, is what makes sure that the cell structure is kept equivariant. Then we reach the claim via the same rank arguments as above.
\item Fortunately, we cannot get into the situation of a circuit of arrows, because the discrete gradient vector field constructed by Algorithm~\ref{ALGdiscretevectorfield} is admissible.\qedhere
\end{itemize}
\end{proof}

\begin{remark}
  Replacing the $\ell$-fusion control condition in Condition~\ref{cancelling-conditions} by the stronger condition that the two involved cell stabilizers are isomorphic recovers Freij's approach \cite{Freij} to equivariant discrete Morse theory. This also provides an analogous algorithm for computing Bredon homology (useful if we want to continue the authors' Bredon homology computations~\cite{psl4z} on larger cell complexes).
\end{remark}

\subsection{Implementation and experimental evaluation}
\label{sec:experiments}

Theorem \ref{correctness-of-algorithm} allows us to simplify the beginning of the equivariant spectral sequence, using Algorithm~\ref{ALGdiscretevectorfield}. The authors' implementation has been released in the HAP package of GAP~\cite{HAP}, reducing the torsion subcomplexes along the discrete vector field~\cite{Bui}. On the $2$-torsion subcomplex of ${\rm PSL}_4(\mathbb{Z})$, this brings down the runtime of the reduction procedure from a matter of days to a matter of seconds or at most minutes. That means, we could still do the $2$-torsion calculation for ${\rm PSL}_4(\mathbb{Z})$ with the previous reduction procedure. But however, we choose this as our example for the rest of this paper, because for readers who want to tackle computations which are out of reach of the obsolete reduction procedure (for instance, on the $2$-torsion subcomplexes of GL$_3$ over imaginary quadratic integers), our example computation can show how to overcome the difficulties which come next, once the reduction has been achieved.

\section{Showcase example: The 2-torsion subcomplex for \texorpdfstring{${\rm PSL}_4(\mathbb{Z})$}{PSL4Z}}
\label{sec:torsionsubcomplex}

The authors have run a machine computation, starting with a Vorono\"i cell complex with an action of ${\rm PSL}_4(\mathbb{Z})$, constructed by Dutour Sikiric~\cite{Dutour:Ellis:Schuermann}, and available in the computer algebra system GAP~\cite{GAP} as part of the package HAP~\cite{HAP}.

More precisely, ${\rm PSL}_4(\mathbb{Z})$ acts on a certain space of quadratic forms, and there is an equivariant retraction to Ash’s well-rounded retract~\cite{Ash}. On the latter space, $\op{PSL}_4(\Z)$ acts co-compactly, and a suitable form of Vorono\"i’s algorithm yields an explicit (finite) cell structure with cell stabilizers and computable quotient space, as described by Ellis, Dutour Sikiric and Sch\"urmann in their paper on the integral homology of ${\rm PSL}_4(\mathbb{Z})$ \cite{Dutour:Ellis:Schuermann}.

For the purpose of computing the cohomological $2$-torsion, we want to extract the $2$-torsion subcomplex from this Vorono\"i cell complex.
As mentioned at the beginning of Section~\ref{sec:morse}, we first need to make a subdivision to achieve a rigidity property and get the $2$-torsion subcomplex to be a cell complex -- this rigidity property is lacking on the original Vorono\"i cell complex of \cite{Dutour:Ellis:Schuermann}. We use “virtually simplicial subdivision”~\cite{psl4z} for this purpose. Its implementation~\cite{Bui} also provides code to extract the $2$-torsion subcomplex.

Next, we reduce the latter using the methods of Section~\ref{sec:morse}, which are also implemented in the same sub-package of HAP~\cite{HAP}. The resulting reduced $2$-torsion subcomplex is shown in Figure~\ref{reduced2torsionSubcomplex}.

\begin{figure}[htp]
 \caption{A reduced $2$-torsion subcomplex for PSL$_4(\Z)$. It is subject to the edge identifications $e_9 = e_9' = e_9'' = e_9'''$ and $e_8 = e_8'$. The cell stabilizers are provided in Table~\ref{stabilizers-and-their-cohomology-rings}.}
 \label{reduced2torsionSubcomplex}
\begin{tikzpicture}
    \node[shape=circle,draw=black] (v5second) at (0,-0.5) {$v_5''$};
    \node[shape=circle,draw=black] (v3) at (0,2.5) {$v_3$};
    \node[shape=circle,draw=black] (v5) at (2.5,4) {$v_5$};
    \node[shape=circle,draw=black] (v4) at (7.5,4) {$v_4$};
    \node[shape=circle,draw=black] (v1) at (2.5,1) {$v_1$};
    \node[shape=circle,draw=black] (v3prime) at (2.5,-2) {$v_3'$};
    \node[shape=circle,draw=black] (v2) at (5,2.5) {$v_2$} ;
    \node[shape=circle,draw=black] (v5prime) at (5,-0.5) {$v_5'$} ;

    \node (t1) at (4.0,1) {$t_1$};
    \node (t4) at (1.0,1) {$t_4$};
    \node (t2) at (3.5,2.5) {$t_2$};
    \node (t3) at (1.5,2.5) {$t_3$};
    \node (t5) at (3.5,-0.5) {$t_5$};
    \node (t6) at (1.5,-0.5) {$t_6$};
    \path [-](v2) edge node[left] {$e_4$} (v4);
    \path [-](v5second) edge node[left] {$e_9'$} (v3);
    \path [-](v3) edge node[left] {$e_9$} (v5);
    \path [-](v5second) edge node[left] {$e_5$} (v1);
    \path [-](v1) edge node[right] {$e_6$} (v5);
    \path [-](v5second) edge node[right] {$e_9'''$} (v3prime);
    \path [-](v1) edge node[left] {$e_3$} (v3prime);
    \path [-](v1) edge node[right] {$e_1$} (v2);
    \path [-](v5) edge node[right] {$e_8'$} (v2);
    \path [-](v3prime) edge node[right] {$e_9''$} (v5prime);
    \path [-](v2) edge node[right] {$e_8$} (v5prime);
    \path [-](v1) edge node[left] {$e_7$} (v5prime);
    \path [-](v1) edge node[right] {$e_2$} (v3);
\end{tikzpicture}
\end{figure}

\subsection{The cell stabilizers and their cohomology rings}

With a GAP script~\cite{supplementary-material}, we generate a SAGE~\cite{Sage} script  containing matrix representations of the cell stabilizer maps from that reduced $2$-torsion-subcomplex, with which the individual ring homomorphisms induced between the cohomology rings of cell stabilizers can be computed, using King's package~\cite{King} in SAGE version 9.6beta4.
Its output is provided in the supplementary materials compiled for the present paper~\cite{supplementary-material}, and yields the cohomology ring presentations in Table~\ref{stabilizers-and-their-cohomology-rings}.

\begin{table}
\caption{Cell stabilizers in the reduced $2$-torsion subcomplex for PSL$_4(\Z)$, and their cohomology rings.}
\label{stabilizers-and-their-cohomology-rings}
\begin{mdframed}
Denoting $\Gamma :=$ PSL$_4(\Z)$ and $\Gamma_\sigma := \{\gamma \in \Gamma : \gamma\cdot\sigma = \sigma\}$ the stabilizer of a cell $\sigma$, we observe the following stabilizers and their cohomology rings for the cells in Figure~\ref{reduced2torsionSubcomplex}.

\footnotesize
\underline{Vertex stabilizers:}

$\Gamma_{v_1} \cong G_1$ where $G_1$ is \verb!SmallGroup(288,1026)! in GAP's SmallGroups-library, given by the non-trivial group extension
$1 \to \Afour \times \Afour \to G_1 \to \Ctwo \to 1$,

$\Homol^*(\Gamma_{v_1};\thinspace \F_2) \cong \F_2[c_{2,0}, c_{2,1}, b_{6,5}, b_{1,0}, b_{3,1}, b_{3,3}]/\langle
b_{1,0}\cdot b_{3,1}, b_{1,0}\cdot b_{3,3},  b_{6,5}\cdot b_{1,0},  b_{6,5}\cdot  b_{3,1}\cdot  b_{3,3}+ b_{6,5}^2+c_{2,1}^3\cdot b_{3,1}^2+ c_{2,0}^3\cdot b_{3,3}^2
\rangle$.

$\Gamma_{v_2} \cong \Sfour$,
$\Homol^*(\Gamma_{v_2};\thinspace \F_2) \cong \F_2[\beta_1, \beta_2, \beta_3]/\langle
\beta_1\cdot \beta_3
\rangle$.

$\Gamma_{v_3} \cong \CtwoCube$,
$\Homol^*(\Gamma_{v_3};\thinspace \F_2) \cong \F_2[\lambda_1, \mu_1, \nu_1]$.

$\Gamma_{v_4} \cong \Sfour$,
$\Homol^*(\Gamma_{v_4};\thinspace \F_2) \cong \F_2[\gamma_1, \gamma_2, \gamma_3]/\langle
\gamma_1\cdot \gamma_3
\rangle$.

$\Gamma_{v_5} \cong G_2$ where $G_2$ is \verb!SmallGroup(96,227)! in GAP's SmallGroups-library, obtained from the non-trivial extensions $1 \to G_0 \to G_2 \to \Ctwo \to 1$
and $1 \to (\Z/2\Z)^4 \to G_0 \to \Z/3\Z \to 1$,

$\Homol^*(\Gamma_{v_5};\thinspace \F_2) \cong \F_2[\beta_{2,1}, \gamma_{2,0}, \gamma_{2,2}, \beta_{1,0}, \beta_{3,0}, \beta_{3,1}, \beta_{3,3}, \beta_{3,5}]/\langle
\beta_{2,1}\cdot \beta_{1,0},
 \beta_{1,0}\cdot \beta_{3,0},
 \beta_{1,0}\cdot \beta_{3,1},
 \beta_{1,0}\cdot \beta_{3,3},
 \beta_{1,0}\cdot \beta_{3,5},
 \beta_{2,1}\cdot \beta_{3,3}+\beta_{2,1}\cdot \beta_{3,1}+\gamma_{2,2}\cdot \beta_{3,3}+\gamma_{2,2}\cdot \beta_{3,1}+\gamma_{2,2}\cdot \beta_{3,0}+\gamma_{2,0}\cdot \beta_{3,3}+\gamma_{2,0}\cdot \beta_{3,0},
 \beta_{2,1}\cdot \beta_{3,5}+\beta_{2,1}\cdot \beta_{3,1}+\gamma_{2,2}\cdot \beta_{3,5}+\gamma_{2,2}\cdot \beta_{3,3}+\gamma_{2,2}\cdot \beta_{3,1}+\gamma_{2,0}\cdot \beta_{3,5}+\gamma_{2,0}\cdot \beta_{3,3}+\gamma_{2,0}\cdot \beta_{3,0},
 \beta_{3,1}\cdot \beta_{3,3}+\beta_{3,0}\cdot \beta_{3,1}+\beta_{3,0}^2+\beta_{2,1}^3+\beta_{2,1}^2\cdot \gamma_{2,2}+\beta_{2,1}^2\cdot \gamma_{2,0},
 \beta_{3,1}\cdot \beta_{3,5}+\beta_{3,0}\cdot \beta_{3,3}+\beta_{3,0}^2+\beta_{2,1}^3,
 \beta_{3,3}^2+\beta_{3,0}\cdot \beta_{3,5}+\beta_{3,0}^2+\beta_{2,1}^3+\beta_{2,1}^2\cdot \gamma_{2,2}
\rangle$.

\underline{Edge stabilizers:}

$\Gamma_{e_1} \cong \Cfour$,
$\Homol^*(\Gamma_{e_1};\thinspace \F_2) \cong \F_2[c_1, c_2]/\langle
c_1^2
\rangle$.

$\Gamma_{e_2} \cong \CtwoSq$,
$\Homol^*(\Gamma_{e_2};\thinspace \F_2) \cong \F_2[x_1, y_1]$.

$\Gamma_{e_3} \cong \CtwoSq$,
$\Homol^*(\Gamma_{e_3};\thinspace \F_2) \cong \F_2[b_1, v_1]$.

$\Gamma_{e_4} \cong \Deight$,
$\Homol^*(\Gamma_{e_4};\thinspace \F_2) \cong \F_2[r_1, t_1, t_2]/\langle
t_1^2+r_1\cdot t_1
\rangle$.

$\Gamma_{e_5} \cong \Ctwo$,
$\Homol^*(\Gamma_{e_5};\thinspace \F_2) \cong \F_2[\widetilde{b_1}]$.

$\Gamma_{e_6} \cong \Deight$,
$\Homol^*(\Gamma_{e_6};\thinspace \F_2) \cong \F_2[a_1, s_1, s_2]/\langle
a_1\cdot s_1
\rangle$.

$\Gamma_{e_7} \cong \Deight$,
$\Homol^*(\Gamma_{e_7};\thinspace \F_2) \cong \F_2[a_1',w_1,w_2]/\langle
a_1'\cdot w_1
\rangle$.

$\Gamma_{e_8} \cong \Ctwo$,
$\Homol^*(\Gamma_{e_8};\thinspace \F_2) \cong \F_2[\widetilde{a_1}]$.

$\Gamma_{e_9} \cong \CtwoCube$,
$\Homol^*(\Gamma_{e_9};\thinspace \F_2) \cong \F_2[x_1', y_1', z_1]$.

\underline{Triangle stabilizers:}

$\Gamma_{t_1} = \Gamma_{t_2} \cong \Ctwo$,
$\Homol^*(\Gamma_{t_1};\thinspace \F_2) \cong \F_2[\widetilde{a_1}]$.

$\Gamma_{t_3} \cong \CtwoSq$,
$\Homol^*(\Gamma_{t_3};\thinspace \F_2) \cong \F_2[\widetilde{x_1}, \widetilde{y_1}]$.

$\Gamma_{t_4} = \Gamma_{t_6} \cong \Ctwo$,
$\Homol^*(\Gamma_{t_4};\thinspace \F_2) \cong \F_2[\widetilde{b_1}]$.

$\Gamma_{t_5} \cong \CtwoSq$,
$\Homol^*(\Gamma_{t_3};\thinspace \F_2) \cong \F_2[\widetilde{u_1}, \widetilde{v_1}]$.
\normalsize
\end{mdframed}
\end{table}

\begin{remark}
Note that we merge the $2$-cells $t_1$ and $t_2$ along their common edge $e_8$ (these three cells have isomorphic stabilizers). But we may not merge
the $2$-cells $t_4$ and $t_6$ along their common edge $e_5$: Also these three cells have isomorphic stabilizers, but there are two copies of the edge $e_9$ adjacent, and two further copies of $e_9$ non-adjacent to $t_4$ and $t_6$. This forbidden merge would prevent the diagram in Figure~\ref{v3} from commuting. However, in order to simplify the diagram in Figure~\ref{v1}, we think of $t_4$ and $t_6$ being virtually merged just when chasing through that diagram.
\end{remark}

Then we assemble these cohomology ring maps, using the new HAP functions written for this purpose by Ellis, and available in the current version of HAP~\cite{HAP}, with which commutative diagrams of finite groups can be used to induce commutative diagrams on mod $\ell$ (especially at $\ell=2$) cohomology of those groups, keeping track of the cup product structure. Before the new functionality became available, keeping track of the cup product structure in HAP was possible only when the finite groups were ``prime power'' / ``$\ell$-groups'': all their elements had to be of order a power of the same prime. This restriction has been lifted with Ellis' new functions. We enter all those of the commutative diagrams of the $2$-torsion subcomplex of ${\rm PSL}_4(\mathbb{Z})$ which have to commute simultaneously. As this allows it to read out in HAP what happens in degrees $1$ to $4$ of the cup product structure~\cite{supplementary-material}, we are able to assemble the $d_1$ differential from our information about the individual cohomology ring maps obtained with King's package~\cite{King} - see our collection of induced maps (19 handwritten pages) on the generators of the stabilizer's cohomology rings, which make all the 16 diagrams (on 14 handwritten pages) commute~\cite{supplementary-material}.
We then use a Pari/GP~\cite{PariGP} script, developed with the help of Allombert, to check the equation $d_1^{1,*} \circ d_1^{0,*}\equiv 0 \mod 2$ which results from the commutativity of the diagrams. This Pari/GP script~\cite{supplementary-material} also outputs these diagrams (Figures~\ref{v1}, \ref{v3} and \ref{v5}) and the $d_1$ differential (Table~\ref{d1}).
The assembled $d_1$ differential (Table~\ref{d1}) is then manually coded into a SAGE script~\cite{supplementary-material}
which computes from it the $E_2$ page in a range of feasible degrees, as shown in Table~\ref{E2page}.

\section{The \texorpdfstring{$d_1$}{d1} differential for \texorpdfstring{${\rm PSL}_4(\mathbb{Z})$}{PSL4Z}} \label{sec:d1}

The bulk of the six years which this paper took to mature, went into establishing the $d_1$ differential (Table~\ref{d1}). King's SAGE package~\cite{King} allows to induce the individual cohomology ring maps from the cell stabilizer maps, but the relations between them are getting lost as the ring presentations are obtained by calls to the SINGULAR computer algebra system. Various ideas on how to ensure coherence on the choices of ring presentations were explored, until finally Ellis did extend the functionality of his GAP package HAP~\cite{HAP}, such that we can trace precisely what happens with the cohomology classes of degrees 1, 2 and 3 (see the 19 handwritten pages in \cite{supplementary-material}). This removes all ambiguities.

\begin{table}\caption{The $d_1$ differential on the generators of the cohomology rings of the cell stabilizers, referring to the ring presentations given in Table~\ref{stabilizers-and-their-cohomology-rings}.}
\label{d1}
\small
$$\begin{array}{|c|c|c|c|c|}\hline &&&&\\ d_1^{0,*} & v_1: c_{2,0}, c_{2,1}, b_{6,5}, b_{1,0}, b_{3,1}, b_{3,3} & v_2: \beta_1, \beta_2, \beta_3 & v_3: \lambda_1, \mu_1, \nu_1 & v_4: \gamma_1, \gamma_2, \gamma_3 \\ \hline
e_1: &
0, c_2, 0, c_1, 0, c_1\cdot c_2 &
c_1, c_2, c_1\cdot c_2  &
0, 0, 0
&
0, 0, 0
\\
e_2: &
x_1^2 + y_1\cdot x_1 + y_1^2, 0, 0, x_1 + y_1, 0, 0
&
0, 0, 0
&
x_1, x_1, y_1
&
0, 0, 0
\\
e_3: &
b_1^2 + v_1\cdot b_1 + v_1^2, v_1\cdot b_1 + v_1^2, 0, b_1, 0, 0
&
0, 0, 0
&
b_1, v_1, 0
&
0, 0, 0
\\
e_4: &
0, 0, 0, 0, 0, 0
&
t_1 + t_1, t_1^2 + t_2, t_2\cdot t_1  &
0, 0, 0
&
t_1, t_2, t_2\cdot t_1 + t_2\cdot t_1
\\
e_5: &
\widetilde{b_1}^2, 0, 0, \widetilde{b_1}, 0, 0
&
0, 0, 0
&
0, 0, 0
&
0, 0, 0
\\
e_6: &
s_1^2 + s_2, a_1^2, a_1^4\cdot s_2 , s_1, a_1\cdot s_2 , 0
&
0, 0, 0
&
0, 0, 0
&
0, 0, 0
\\
e_7: &
w_2, w_2 + a_1'^2 + w_1^2, a_1'^4\cdot w_2 , w_1, a_1'\cdot w_2 , a_1'\cdot w_2  &
0, 0, 0
&
0, 0, 0
&
0, 0, 0
\\
e_9: &
0, 0, 0, 0, 0, 0
&
0, 0, 0
&
y_1', z_1, z_1 + x_1'
&
0, 0, 0
\\
\hline \end{array}$$

$$\begin{array}{|c|c|}\hline &\\ d_1^{0,*} & v_5: \beta_{2,1}, \gamma_{2,0}, \gamma_{2,2}, \beta_{1,0}, \beta_{3,0}, \beta_{3,1}, \beta_{3,3}, \beta_{3,5} \\ \hline
e_1: &
0, 0, 0, 0, 0, 0, 0, 0
\\
e_2: &
0, 0, 0, 0, 0, 0, 0, 0
\\
e_3: &
0, 0, 0, 0, 0, 0, 0, 0
\\
e_4: &
0, 0, 0, 0, 0, 0, 0, 0
\\
e_5: &
\widetilde{b_1}^2, 0, 0, 0, \widetilde{b_1}^3, \widetilde{b_1}^3, 0, 0
\\
e_6: &
s_1^2, s_2, s_2 + a_1^2, a_1, s_1^3 + s_2\cdot s_1 , s_1^3 + s_2\cdot s_1 , 0, s_2\cdot s_1  \\
e_7: &
w_1^2, w_2 + w_1^2, a_1'^2, a_1', w_1\cdot w_2 , w_1^3, 0, w_1\cdot w_2 + w_1^3
\\
e_9: & \begin{pmatrix}
y_1'^2 + (z_1 + x_1')\cdot y_1' + x_1'\cdot z_1 + x_1'^2,\\ z_1\cdot y_1' + x_1'\cdot z_1,\\ z_1^2 + x_1'\cdot z_1 ,\\ 0,\\ y_1'^3 + x_1'\cdot y_1'^2 + z_1^2\cdot y_1' + x_1'^2\cdot z_1 + x_1'^3,\\ y_1'^3 + (z_1 + x_1')\cdot y_1'^2 + x_1'^2\cdot z_1 + x_1'^3,\\ 0,\\ z_1\cdot y_1'^2 + z_1^2\cdot y_1' + x_1'\cdot z_1^2 + x_1'^2\cdot z_1 \end{pmatrix}
\\
\hline \end{array}$$

$$\begin{array}{|l|c|c|c|c|c|c|c|l|}\hline &&&&&&&&\\ d_1^{1,*} & e_1: & e_2: & e_3: & e_4: & e_5: & e_6: & e_7: & e_9: \\
 &  c_{1}, c_{2} & x_1, y_1 & b_1, v_1 &  r_1, t_1, t_2 &  \widetilde{b_1} &  a_1, s_1, s_2 &  a_1', w_1, w_2 & x_1', y_1', z_1 \\ \hline
t_1 \cup t_2: &
0, 0
&
0, 0
&
0, 0
&
0, 0, 0
&
0
&
\widetilde{a_1}, 0, 0
&
\widetilde{a_1}, 0, 0
&
0, 0, 0
\\
t_3: &
0, 0
&
\widetilde{x_1}, \widetilde{y_1}
&
0, 0
&
0, 0, 0
&
0
&
0, \widetilde{x_1} + \widetilde{y_1}, \widetilde{y_1}\cdot \widetilde{x_1}
&
0, 0, 0
&
\widetilde{x_1} + \widetilde{y_1}, \widetilde{x_1}, \widetilde{x_1}
\\
t_4: &
0, 0
&
0, \widetilde{b_1}
&
0, 0
&
0, 0, 0
&
\widetilde{b_1}
&
0, 0, 0
&
0, 0, 0
&
\widetilde{b_1}, 0, 0
\\
t_5: &
0, 0
&
0, 0
&
\widetilde{u_1}, \widetilde{v_1}
&
0, 0, 0
&
0
&
0, 0, 0
&
0, \widetilde{u_1}, \widetilde{u_1}^2 + \widetilde{v_1}\cdot \widetilde{u_1} + \widetilde{v_1}^2
&
\widetilde{v_1}, \widetilde{u_1}, \widetilde{v_1}
\\
t_6: &
0, 0
&
0, 0
&
\widetilde{b_1}, 0
&
0, 0, 0
&
\widetilde{b_1}
&
0, 0, 0
&
0, 0, 0
&
0, \widetilde{b_1}, 0
\\
\hline \end{array}$$
\end{table}

\subsection{Sanity check: Commutative diagrams verified by the $d_1$ differential}

The vanishing of the concatenation $d_1^{1,q} \circ d_1^{0,q} \equiv 0 \mod 2$ is ensured by the commutative diagrams of ring generators of the cohomology rings of the finite cell stabilizer groups, displayed in Figures~\ref{v1}, \ref{v3} and ~\ref{v5}.

\begin{figure}[htp]
 \caption{Commutative diagram of mappings of the cohomology ring generators of the 1st vertex stabilizer.} \label{v1}
\[ \begin{tikzcd}  \begin{pmatrix}
0   \\
\widetilde{a_1}^2   \\
0   \\
0   \\
0   \\
0   \\
\end{pmatrix}   & \begin{pmatrix}
w_2   \\
w_2 + a_1'^2 + w_1^2   \\
a_1'^4\cdot w_2   \\
w_1   \\
a_1'\cdot w_2   \\
a_1'\cdot w_2   \\
\end{pmatrix} \arrow[l,"t_1\cup t_2"] \arrow[r,"t_5"] & \begin{pmatrix}
\widetilde{u_1}^2 + \widetilde{v_1}\cdot \widetilde{u_1} + \widetilde{v_1}^2   \\
\widetilde{v_1}\cdot \widetilde{u_1} + \widetilde{v_1}^2   \\
0   \\
\widetilde{u_1}   \\
0   \\
0   \\
\end{pmatrix} \\ \begin{pmatrix}
s_1^2 + s_2   \\
a_1^2   \\
a_1^4\cdot s_2   \\
s_1   \\
a_1\cdot s_2   \\
0   \\
\end{pmatrix} \arrow[d,swap,"t_3"] \arrow[u,"t_1\cup t_2"] & \begin{pmatrix}
c_{2,0}  \\ c_{2,1}   \\ b_{6,5}   \\ b_{1,0}   \\ b_{3,1}   \\ b_{3,3}
\end{pmatrix}  \arrow[u,swap,"e_7"] \arrow[l,swap,"e_6"] \arrow[r,"e_3"] \arrow[d,"e_2"] &
 \begin{pmatrix}
b_1^2 + v_1\cdot b_1 + v_1^2 \\
v_1\cdot b_1 + v_1^2 \\
0 \\
b_1 \\
0 \\
0 \\
\end{pmatrix}  \arrow[d,swap,"t_4\cup t_6"] \arrow[u,swap,"t_5"]  \\ \begin{pmatrix}
\widetilde{x_1}^2 + \widetilde{y_1}\cdot \widetilde{x_1} + \widetilde{y_1}^2 \\
0 \\
0 \\
\widetilde{x_1} + \widetilde{y_1} \\
0 \\
0 \\
\end{pmatrix}  & \begin{pmatrix}
x_1^2 + y_1\cdot x_1 + y_1^2 \\
0 \\
0 \\
x_1 + y_1 \\
0 \\
0 \\
\end{pmatrix} \arrow[l,"t_3"] \arrow[r,swap,"t_4 \cup t_6"] & \begin{pmatrix}
\widetilde{b_1}^2 \\
0 \\
0 \\
\widetilde{b_1} \\
0 \\
0 \\
\end{pmatrix} \end{tikzcd} \]
\end{figure}

\begin{figure}[htp]
 \caption{Commutative diagram of mappings of the cohomology ring generators of the 3rd vertex stabilizer.} \label{v3}
\[ \begin{tikzcd}  \begin{pmatrix}
\widetilde{u_1}   \\
\widetilde{v_1}   \\
0   \\
\end{pmatrix}   & \begin{pmatrix}
b_1   \\
v_1   \\
0   \\
\end{pmatrix} \arrow[l,swap,"t_5"] \arrow[r,"t_6"] & \begin{pmatrix}
\widetilde{b_1}   \\
0   \\
0   \\
\end{pmatrix} \\ \begin{pmatrix}
y_1'   \\
z_1   \\
z_1 + x_1'   \\
\end{pmatrix} \arrow[d,swap,"t_3"] \arrow[u,swap,"t_5"] & \begin{pmatrix}
\lambda_1   \\ \mu_1   \\ \nu_1  \\
\end{pmatrix} \arrow[l,swap,"e_9"] \arrow[u,"e_3"] \arrow[r,"e_9"] \arrow[d,"e_2"] & \begin{pmatrix}
y_1'   \\
z_1   \\
z_1 + x_1'   \\
\end{pmatrix} \arrow[u,"t_6"] \arrow[d,"t_4"] \\ \begin{pmatrix}
\widetilde{x_1}   \\
\widetilde{x_1}   \\
\widetilde{y_1}   \\
\end{pmatrix}  & \begin{pmatrix}
x1 \\
x1 \\
y1 \\
\end{pmatrix}  \arrow[l,"t_3"] \arrow[r,swap,"t_4"] & \begin{pmatrix}
0 \\
0 \\
\widetilde{b_1} \\
\end{pmatrix}   \end{tikzcd} \]
 \end{figure}

\begin{figure}[htp]
 \caption{Commutative diagram of mappings of the cohomology ring generators of the 5th vertex stabilizer.} \label{v5}
\begin{tikzcd}
& \begin{pmatrix} \widetilde{u_1}^2 \\
\widetilde{v_1}^2 +\widetilde{u_1}\cdot \widetilde{v_1} \\ 0\\ 0\\
\widetilde{u_1}^3+\widetilde{u_1}^2\cdot \widetilde{v_1} +\widetilde{u_1} \cdot\widetilde{v_1}^2 \\
\widetilde{u_1}^3 \\ 0 \\
\widetilde{u_1}^2\cdot \widetilde{v_1} +\widetilde{u_1}\cdot\widetilde{v_1}^2   \end{pmatrix}
 \\
\begin{pmatrix} x_1'^2+x_1'\cdot y_1'+x_1'\cdot z_1+y_1'^2+y_1'\cdot z_1\\
 x_1'\cdot z_1+y_1'\cdot z_1\\
 x_1'\cdot z_1+z_1^2\\
 0 \\
x_1'^3+x_1'^2\cdot z_1+y_1'^3+y_1'\cdot z_1^2+x_1'\cdot y_1'^2\\
x_1'\cdot y_1'^2+y_1'^3+y_1'^2\cdot z_1+x_1'^3+x_1'^2\cdot z_1\\
0\\
x_1'^2\cdot z_1+x_1'\cdot z_1^2+y_1'^2\cdot z_1+y_1'\cdot z_1^2
\end{pmatrix}
 \arrow[dd,swap,"t_3"] \arrow[ur,"t_5"] && \begin{pmatrix} w_1^2 \\ w_1^2 +w_2\\ a_1'^2 \\
a_1' \\ w_2\cdot w_1 \\ w_1^3 \\ 0 \\ w_1^3+w_2\cdot w_1                                \end{pmatrix}
\arrow[dd,swap,"t_1\cup t_2"] \arrow[ul,swap,"t_5"] \\
& \begin{pmatrix}\beta_{2,1}\\ \gamma_{2,0}\\\gamma_{2,2} \\ \beta_{1,0} \\ \beta_{3,0} \\ \beta_{3,1}  \\ \beta_{3,3} \\ \beta_{3,5}
  \end{pmatrix}
 \arrow[ul,swap,"e_9"] \arrow[ur,"e_7"] \arrow[dd,"e_6"] \\
\begin{pmatrix}\widetilde{x_1}^2+\widetilde{y_1}^2 \\ \widetilde{x_1}\cdot \widetilde{y_1}\\ \widetilde{x_1}\cdot \widetilde{y_1}\\ 0 \\ \widetilde{x_1}^3 +\widetilde{y_1}^3 \\ \widetilde{x_1}^3 +\widetilde{y_1}^3 \\ 0\\ \widetilde{x_1}^2\cdot \widetilde{y_1} +\widetilde{x_1}\cdot \widetilde{y_1}^2
\end{pmatrix}
  &&
\begin{pmatrix}  0\\ 0\\ \widetilde{a_1}^2 \\ \widetilde{a_1} \\ 0\\ 0\\ 0\\ 0
\end{pmatrix}
 \\
& \begin{pmatrix} s_1^2 \\ s_2 \\ s_2+a_1^2\\ a_1 \\ s_1^3+s_2\cdot s_1 \\ s_1^3+s_2\cdot s_1 \\ 0 \\ s_2\cdot s_1
  \end{pmatrix}
 \arrow[ul,"t_3"] \arrow[ur,swap,"t_1 \cup t_2"]
\end{tikzcd}
\end{figure}

\section{Spectral sequence evaluation and mod 2 Farrell--Tate cohomology of \texorpdfstring{${\rm PSL}_4(\mathbb{Z})$}{PSL4Z}}
\label{sec:farrelltate}

From the $E_2$-page shown in Table~\ref{E2page}, we see that the $d_2$ differential must be trivial, and hence the spectral sequence degenerates at that stage.

\begin{table}[htp]
\caption{{The $E_2$-page of our equivariant spectral sequence.}}
\label{E2page}
\normalsize
In all observed degrees $q\geq 2$, the differential $d_1^{0,q}$ is surjective,
and hence at $q\geq 2$, the $E_2$-page is concentrated in the single column $p = 0$.
In that column, we record the dimensions over $\F_2$ for $E_2^{0,q}$ for $7 \leq q \leq 42$ in Table~\ref{7-42}.
In all degrees $q \geq 0$, the $E_2$-page is concentrated in three columns,
\\ which we display here for $q \leq 6$: 
\\
$ \xymatrix{
q =  6  & & \F_2^{
27 }
 & & 0 & &
0
\\
q =  5  & & \F_2^{
17}
 & & 0 & &
0
\\
q =  4  & & \F_2^{
7 }
 & & 0 & &
0
\\
q =  3  & & \F_2^{
7 }
 & & 0 & &
0
\\
q =  2  & & \F_2^{
3 }
 & & 0 & &
0
\\
q =  1  & & 0
 & & \F_2 & &
0
\\
q=0 & & \F_2 & & 0 & & \F_2
\\
{} && p = 0 && p=1 && p=2
\\
}$
\normalsize
\end{table}

From the $E_2$-page calculated in Table~\ref{E2page}, we deduce our results on the mod-2 Farrell--Tate cohomology of {PSL}$_4(\Z)$, presented in Theorem~\ref{result} in the Introduction.

\section{A Hilbert--Poincar\'e series computation}
\label{sec:hpseries}

In this section, we compute the Hilbert--Poincar\'e series for the mod 2 cohomology of $\op{PSL}_4(\Z)$, both as a sanity check as well as a finitary encoding of all the mod 2 Betti numbers. 

We first recall the Hilbert--Poincar\'e series for the stabilizer groups of cells appearing in the reduced 2-torsion subcomplex depicted in Figure~\ref{reduced2torsionSubcomplex}, see Table~\ref{stabilizers-and-their-cohomology-rings}. The series are as follows, for the groups $G_1$ and $G_2$, the Hilbert--Poincar\'e series can be computed using SageMath or Macaulay2 from the algebra presentation in Table~\ref{stabilizers-and-their-cohomology-rings}: 
\begin{align*}
  {\rm HP}(G_1,T)&=\frac{1-2T^4+2T^{10}-T^{12}}{(1-T^6)(1-T^3)^2(1-T^2)^2(1-T)}\\
  {\rm HP}(G_2,T)&=\frac{1-T^3-4T^4-2T^5+3T^6+9T^7+3T^8-7T^9-6T^{10}+4T^{12}+T^{13}-T^{14}}{(1-T^3)^4(1-T^2)^3(1-T)}\\
  {\rm HP}(\Sfour,T)&=\frac{1-T^4}{(1-T^3)(1-T^2)(1-T)}\\
  {\rm HP}((\Ctwo)^n,T)&=\frac{1}{(1-T)^n}
\end{align*}
Note also that ${\rm HP}(\Deight,T)={\rm HP}((\Ctwo)^2,T)$ and ${\rm HP}(\Cfour,T)={\rm HP}(\Ctwo,T)$.

Denoting by $\mathcal{X}$ the reduced 2-torsion subcomplex in Figure~\ref{reduced2torsionSubcomplex}, we can now consider the following power series, whose coefficients are alternating sums of entries of the equivariant spectral sequence:
\begin{align*}
  \sum_q\left(\sum_{\sigma\in \mathcal{X}}(-1)^{\dim\sigma}\dim_{\F_2}\Homol^q(\Gamma_\sigma;\thinspace\F_2)\right)T^q
  &=\sum_q\left(\sum_p(-1)^p\dim_{\F_2}E^{p,q}_1\right)T^q\\
  &=\sum_q\left(\sum_p(-1)^p\dim_{\F_2}E^{p,q}_2\right)T^q.
\end{align*}
The first equality is by definition of the equivariant spectral sequence, and the second equality is the usual equality for Euler characteristics for the complexes formed by the $d_1$-differential on the $E_1$-page. From the cell structure of the reduced 2-torsion subcomplex, the above expression is given by Hilbert--Poincar\'e series of stabilizers as follows:
\begin{align*}
  &\sum_q\left(\sum_p(-1)^p\dim_{\F_2}E^{p,q}_2\right)T^q\\
  =&{\rm HP}(G_1,T)+{\rm HP}(G_2,T)+2{\rm HP}(\Sfour,T)+{\rm HP}(\Ctwo,T)^3\\
  &-\left(3{\rm HP}(\Ctwo,T)+5{\rm HP}(\Ctwo,T)^2+{\rm HP}(\Ctwo,T)^3\right)\\
  &+4{\rm HP}(\Ctwo,T)+2{\rm HP}(\Ctwo,T)^2\\
  =&{\rm HP}(G_1,T)+{\rm HP}(G_2,T)+2{\rm HP}(\Sfour,T)+{\rm HP}(\Ctwo,T)-3{\rm HP}(\Ctwo,T)^2
\end{align*}
In the second step, the contributions from vertices, edges and triangles are sorted into three different lines. Plugging in the information on Hilbert--Poincar\'e series for group cohomology of stabilizers gathered earlier, we can evaluate this expression; the reduced rational function is the following:
\begin{align*}
  \sum_q\left(\sum_p(-1)^p\dim_{\F_2}E^{p,q}_2\right)T^q&=-\frac{T^8+T^7+2T^6-2T^5-2T^4-3T^3-2T^2+3T-2}{(T^2+T+1)^2(T+1)(T-1)^4}\\
  &=2-T+3T^2+7T^3+7T^4+17T^5+27T^6+31T^7+\cdots
\end{align*}

We see that the computation agrees with the information on the $E_2$-page of the spectral sequence depicted in Table~\ref{E2page}. Since the spectral sequence is concentrated in the column $p=0$ for degrees $q\geq 3$, the Hilbert--Poincar\'e series above describes the group cohomology of $\op{PSL}_4(\Z)$ above the virtual cohomological dimension; the coefficients of the power series expansion agree with the dimensions listed in Table~\ref{7-42}. Adjusting for the low-degree entries outside the column $E^{0,q}_2$, we obtain the Hilbert--Poincar\'e series for the Farrell--Tate cohomology $\Farrell^*(\op{PSL}_4(\Z);\thinspace \F_2)$, whose initial terms agree with the information in Theorem~\ref{result}:
\begin{align*}
  \sum_q\left(\sum_p\dim_{\F_2}E^{p,q-p}_2\right)T^q&=\frac{2T^{11}-T^{10}-4T^9-3T^8+3T^7+5T^6-3T^4+T^3+4T^2-T+1}{(T^2+T+1)^2(T+1)(T-1)^4}\\
  &=1+5T^2+7T^3+7T^4+17T^5+27T^6+31T^7+\cdots
\end{align*}

\section{Implications on Steinberg homology} \label{sec:Steinberg}

In the long exact sequence connecting Farrell--Tate cohomology, group cohomology and Steinberg homology (cf. the chapter on Farrell--Tate cohomology in Brown's book~\cite{Brown}), we insert our results on Farrell--Tate cohomology $\Farrell^* := \Farrell^*(\op{PSL}_4(\Z);\thinspace \F_2)$, the results of Dutour, Ellis and Sch\"urmann~\cite{Dutour:Ellis:Schuermann} on group cohomology $\Homol^* := \Homol^*(\op{PSL}_4(\Z);\thinspace \F_2)$, and on Steinberg homology $\Steinberg_*$ the result $\Steinberg_0 = 0$ of Lee and Szczarba~\cite{Lee:Szczarba}*{theorems 1.3 and 4.1, lemma 4.2}.
Then we can speculate on the Steinberg homology in the other degrees (the numbers inserted here with a question mark assume the case of trivial connection maps in the long exact sequence, and the reader can easily adjust them to cases of non-trivial connection map ranks).
The top line of this long exact sequence demonstrates the compatibility of our results with $\Steinberg_0 = 0$ and the results of Dutour, Ellis and Sch\"urmann.

\begin{center}
\begin{tikzpicture}[descr/.style={fill=white,inner sep=1.5pt}]
        \matrix (m) [
            matrix of math nodes,
            row sep=1em,
            column sep=2.5em,
            text height=1.99ex, text depth=0.75ex
        ]
        {  & \Homol^5 (\dim 17) & \Farrell^5 (\dim 17) &\Steinberg_{0} = 0
         \\& \Homol^4 (\dim 6) & \Farrell^4 (\dim 7) &\Steinberg_{1}(\dim 1?) &  \\
           & \Homol^{3} (\dim 6) & \Farrell^{3} (\dim 7) & \Steinberg_2 (\dim 1?)&  \\
	   & \Homol^2 (\dim 3)& \Farrell^{2} (\dim 5) & \Steinberg_3 (\dim 2?)&\\
           & \Homol^1 (=0)& \Farrell^1 (= 0) &\Steinberg_{4} (=0?)&  \\
           & \Homol^{0} (\dim 1) & \Farrell^{0} (\dim 1) & \Steinberg_5 (= 0?)&  \\
	   & 0 & \Farrell^{-1} & \qquad \Steinberg_6  \qquad{}&\\
        };

        \path[overlay,->, font=\scriptsize,>=latex]
        (m-1-2) edge (m-1-3)
        (m-1-3) edge (m-1-4)
        (m-2-4) edge[out=-355,in=-175] node[descr,yshift=0.9ex] {$0$?} (m-1-2)
        (m-2-2) edge (m-2-3)
        (m-2-3) edge (m-2-4)
        (m-3-4) edge[out=-355,in=-175] node[descr,yshift=0.9ex] {$0$?} (m-2-2)
        (m-3-2) edge (m-3-3)
        (m-3-3) edge (m-3-4)
        (m-4-4) edge[out=-355,in=-175] node[descr,yshift=0.9ex] {$0$?} (m-3-2)
        (m-4-2) edge (m-4-3)
        (m-4-3) edge (m-4-4)
        (m-5-4) edge[out=-355,in=-175]  (m-4-2)
        (m-5-2) edge (m-5-3)
        (m-5-3) edge (m-5-4)
        (m-6-4) edge[out=-355,in=-175]  (m-5-2)
        (m-6-2) edge (m-6-3)
        (m-6-3) edge (m-6-4)
        (m-7-4) edge[out=-355,in=-175]  (m-6-2)
        (m-7-2) edge (m-7-3)
        (m-7-3) edge (m-7-4)
;
\end{tikzpicture}
\end{center}

\end{document}